\documentclass[
a4paper,
11pt,
DIV=14,
toc=bibliography, 
headsepline, 
parskip=half-,
abstract=true,
]{scrartcl}

\usepackage[centertags]{amsmath}
\usepackage[utf8]{inputenc}
\usepackage{enumerate,amsfonts,amssymb,amsthm,amsopn,cite,array,MnSymbol,comment}

\usepackage[usenames]{color}
\definecolor{citegreen}{rgb}{0,0.6,0}
\definecolor{refred}{rgb}{0.8,0,0}
\usepackage[colorlinks, citecolor=citegreen, linkcolor=refred, urlcolor=black]{hyperref}

\usepackage[bbgreekl]{mathbbol}
\DeclareSymbolFontAlphabet{\mathbb}{AMSb}
\DeclareSymbolFontAlphabet{\mathbbm}{bbold}

\title{Mean Curvature Flow and Heegaard Surfaces\\ in Lens Spaces}
\author{Reto Buzano and Sylvain Maillot}

\newtheorem{theo}{Theorem}[section]
\newtheorem{lemma}[theo]{Lemma}

\newtheorem{prop}[theo]{Proposition}
\theoremstyle{definition}
\newtheorem{rem}[theo]{Remark}

\newtheorem{defi}[theo]{Definition}
\newtheorem{claim}{Claim}
\newtheorem{case}{Case}
\numberwithin{equation}{section}

\let\bord\partial
\let\emph\emph

\def\calM{\mathcal{M}}
\def\calG{\mathcal{G}}

\newcommand{\RR}{\mathbb R}
\newcommand{\CC}{\mathbb C}
\newcommand{\ZZ}{\mathbb Z}
\newcommand{\Sph}{\mathbb S}
\newcommand{\bH}{\mathbb H}
\newcommand{\bA}{\mathbb A}
\newcommand{\bb}{\mathbbm b}

\newcommand{\eps}{\varepsilon}
\newcommand{\calD}{\mathcal{D}}
\newcommand{\calC}{\mathcal{C}}
\newcommand{\calX}{\mathcal{X}}
\newcommand{\tth}{\mathrm{th}}
\newcommand{\neck}{\mathrm{neck}}
\newcommand{\trig}{\mathrm{trig}}
\newcommand{\ext}{\mathrm{ext}}

\providecommand{\abs}[1]{\lvert #1\rvert}
\providecommand{\norm}[1]{\lVert #1\rVert}

\DeclareMathOperator{\Ric}{Ric}

\DeclareMathOperator{\Diff}{Diff}
\DeclareMathOperator{\inj}{inj}
\DeclareMathOperator{\Int}{Int}
\date{}

\newcommand\printaddress{{
\setlength{\parindent}{15pt}
\footnotesize~
\par
{\scshape Reto Buzano}
\newline 
Universit\`a degli Studi di Torino, Dipartimento di Matematica,
Torino, Italy 
\newline
\emph{E-mail address:} 
\texttt{reto.buzano@unito.it}
\par
{\scshape Sylvain Maillot}
\newline 
IMAG, Universit\'e de Montpellier, CNRS,
Montpellier, France
\newline
\emph{E-mail address:} 
\texttt{sylvain.maillot@umontpellier.fr}
\par
}}

\begin{document}
\maketitle
\begin{abstract}
We prove that the moduli space of mean convex two-spheres embedded in complete, orientable $3$-dimensional Riemannian manifolds with nonnegative Ricci curvature is path-connected. This result is sharp in the sense that neither of the conditions of (strict) mean convexity, completeness, and nonnegativity of the Ricci curvature can be dropped or weakened. We also study the number of path components of mean convex Heegaard tori, again in ambient manifolds with nonnegative Ricci curvature. We prove that there are always either one or two path components and this number does not only depend on the homotopy type of the ambient manifold. We give a precise characterisation of the two cases and also discuss what happens if the mean convexity condition is weakened to nonnegative mean curvature.
\end{abstract}


\section{Introduction}
Given a closed (smooth) manifold $M^n$, a classical problem in Riemannian geometry is to investigate the homotopy type of the moduli space of positive scalar curvature metrics
\begin{equation*}
\calM_{R>0}(M) := \{ g \text{ Riemannian metric on $M$ with } R_g >0\} / \Diff(M).
\end{equation*} 
This problem has its origins in an over 100 years old result of Weyl~\cite{We16} who showed that for $n=2$, this space is always path-connected. In fact, Rosenberg-Stolz~\cite{RS01} proved that it is contractible. When $n=3$, it is only much more recently that Marques~\cite{Mar12} proved path-connectedness and Bamler-Kleiner~\cite{BK19}  proved contractibility. These results are in sharp contrast to the situation in higher dimensions where the moduli spaces can have very complicated topology -- though their homotopy types are in general unknown. An illustrative example is given by spheres $\Sph^{4k+3}$, $k\geq 1$, for which the moduli spaces are known to have infinitely many path components~\cite{KS93}.

In this article, we are concerned with the following extrinsic version of this problem: given a Riemannian three-manifold $(M^3,g)$ and a closed surface $\Sigma^2$, investigate the homotopy type of the moduli space of mean convex two-sided embeddings of $\Sigma$ into $M$, 
\begin{equation*}
\calM_{H>0}(\Sigma,M) := \{ \Sigma \hookrightarrow M \text{ smooth two-sided embedding with } H >0\} / \Diff(\Sigma).
\end{equation*} 
This extrinsic problem was first studied by the first author in joint work with Haslhofer and Hershkovits in the case where $(M^3,g) = (\mathbb{R}^3, g_{Eucl.})$ is Euclidean space.\footnote{In fact, they studied the moduli spaces of \emph{two-convex} embeddings of hypersurfaces $\Sigma^{n-1}$ into $(\mathbb{R}^n,g_{Eucl.})$, but we only focus on dimension $n=3$ here.} In \cite{BHH21}, they obtained that for $\Sigma = \Sph^2$ the moduli space is path-connected and in \cite{BHH19} they proved that for $\Sigma = T^2 = \Sph^1 \times \Sph^1$ the connected components of the moduli space are in bijective correspondence with the knot classes of closed embedded curves.

The aim of this article is to extend the study of the extrinsic problem to more general ambient three-manifolds $(M^3,g)$. In this context we say that a smoothly embedded, two-sided surface $\Sigma \subset M$ has \emph{positive mean curvature} (respectively \emph{nonnegative mean curvature}) with respect to $g$ if there exists a choice of unit normal vector field $\nu$ on $\Sigma$ such that $H>0$ (resp.~$H\ge 0$) everywhere. In other words, $\Sigma$ has positive mean curvature if the mean curvature vector always points to the same side of $\Sigma$. Note that this latter definition only makes sense if $\Sigma$ is two-sided. We also say that $\Sigma$ is \emph{mean convex} if it has positive mean curvature.

For $\Sigma = \Sph^2$, we obtain the following generalisation of the main result from \cite{BHH21}.

\begin{theo}\label{theo.sphere}
Let $(M,g)$ be a complete, orientable Riemannian three-manifold with nonnegative Ricci curvature. Then the moduli space $\calM_{H>0}(\Sph^2,M)$ of mean convex embedded two-spheres in $(M,g)$ is path-connected.
\end{theo}

\begin{rem}\label{rem.theo1sharp}
Theorem \ref{theo.sphere} does not hold if the condition $H>0$ is weakened to $H\geq 0$ or if $\Ric_g \geq 0$ is weakened to $\Ric_g \geq -C$ for some $C>0$ (or replaced by $1 \geq \Ric_g \geq -\eps$ for some $\eps>0$). Indeed, in both cases, it is easy to construct examples where the moduli space contains two-spheres that do not bound a ball and others that do; thus the moduli space has at least two path components. Moreover, the completeness condition on $M$ cannot be dropped: for example if $M$ is a round three-sphere with a point removed, then every mean convex two-sphere bounds a ball, but for some of them the mean curvature vector points toward the ball while for others it points in the opposite direction. Hence the moduli space again has at least two path components. See Proposition \ref{prop.bounddomain} and Remark \ref{remark.bounddomain} below for more details.
\end{rem}

For $\Sigma = T^2$, we restrict the problem to \emph{Heegaard surfaces} in order to guarantee that $\Sigma$ bounds a solid (unknotted) torus. It is well known that the moduli space of \emph{all} Heegaard tori is path-connected, but surprisingly, in contrast to the Euclidean case, the moduli space of \emph{mean convex} Heegaard tori can have more than one path component. Our main theorem says exactly how many there are depending on the ambient manifold $M$. Before stating the results, we give some definitions and fix notation. 

In his PhD thesis, Heegaard \cite{He98} introduced an intriguingly simple way to construct complicated orientable three-manifolds: taking two handlebodies of the same genus, one can glue their boundaries together by a (possibly very complicated) diffeomorphism. Here, a handlebody of genus $k$ is an orientable three-manifold with boundary obtained from a closed $3$-ball $\bar{B}^3$ by attaching $k$ handles, the process of attaching a handle being understood as removing two disjoint closed disks $D_0$ and $D_1$ in $\partial B^3$ and attaching $D^2 \times [0,1]$ by identifying $D^2 \times \{j\}$ with $D_j$, $j=0,1$. As a consequence of Whitehead's triangulation theorem~\cite{Wh40}, in fact every closed, orientable three-manifold can be obtained in this way. Conversely, if $M$ is a closed, orientable three-manifold, and $\Sigma$ a closed, orientable surface smoothly embedded in $M$, we say that $\Sigma$ is a \emph{Heegaard surface} if the closure of each component of $M\setminus \Sigma$ is a handlebody. Three-manifolds admitting a Heegaard surface of genus one are called \emph{lens spaces}. Let us now describe a general construction of them.

Let $\Sph^1$ and $D^2$ be the unit circle and the unit disk in $\CC$, respectively. A \emph{solid torus} $V$ is a three-manifold diffeomorphic to $\Sph^1\times D^2$. Its \emph{core} is the simply closed curve $\Sph^1\times \{0\}$ and the \emph{longitude} and \emph{meridian} are defined as $\Sph^1 \times \{1\}$ and $\{1\} \times \Sph^1$, respectively. Note that these definitions imply that a diffeomorphism to $\Sph^1\times D^2$ has been fixed, so these concepts are defined only up to isotopy.

Let $p,q$ be coprime integers with $p\ge 0$. The \emph{lens space} $L(p,q)$ is obtained by gluing together two solid tori $V_1$ and $V_2$ via the diffeomorphism $\theta:\partial V_1\to \partial V_2$ defined by 
\begin{equation*}
\theta(u,v)=(u^r v^p, u^s v^q),
\end{equation*}
where $r,s$ are integers such that $rq-sp=1$. Put differently, $L(p,q)$ is the three-manifold obtained by gluing the boundaries of two solid tori in such a way that the meridian of the first torus goes to a curve wrapping around the longitude $p$ times and around the meridian $q$ times on the second torus. We denote by $\gamma_k$ the core of $V_k$ for $k\in\{1,2\}$. By construction, the torus $\Sigma = \bord V_1$ is Heegaard.

Special cases of lens spaces include the $3$-sphere $\Sph^3=L(1,0)$, real projective space $\RR \mathbb{P}^3=L(2,1)$, and $\Sph^1\times \Sph^2=L(0,1)$. On the one hand, the standard product metric on $L(0,1)= \Sph^1\times \Sph^2$ has $\Ric_g \geq 0$. On the other hand, in the case $p\neq 0$, $L(p,q)$ has an alternative description as a quotient of $\Sph^3 \subset \mathbb{C}^2$ by the (finite cyclic) group of isometries generated by
\begin{equation*}
(z_1,z_2) \mapsto (e^{2\pi i/p}z_1, e^{2\pi i q/p}z_2)
\end{equation*}
and therefore carries a Riemannian metric with (constant) positive sectional curvature, thus in particular again with $\Ric_g\geq 0$. The assumption of nonnegative Ricci curvature is therefore quite natural for lens spaces and throughout the article (unless otherwise mentioned), we will always assume that we have fixed a background metric $g$ on $L(p,q)$ with $\Ric_g \geq 0$.

\begin{defi}
We denote by 
\begin{equation*}
\calM_{H>0}(p,q) = \calM_{H>0}(T^2, L(p,q))
\end{equation*} 
the \emph{moduli space of Heegaard tori} in $L(p,q)$ that have positive mean curvature with respect to the background metric $g$ on $L(p,q)$. Similarly, $\calM_{H\geq 0}(p,q)$ denotes the moduli space of Heegaard tori in $L(p,q)$ with nonnegative mean curvature.
\end{defi}

Our main results are the following two theorems.

\begin{theo}\label{theo.mainone}
Let $p,q$ be coprime integers with $p \geq 0$. Fix a metric $g$ on the lens space $L(p,q)$ with nonnegative Ricci curvature. If $q\cong \pm 1 \!\mod p$, then $\calM_{H>0}(p,q)$ is path-connected. Otherwise it has exactly two path-components.
\end{theo}

\begin{rem}
An interesting aspect of Theorem~\ref{theo.mainone} is that it is not true that the homotopy type of the moduli spaces $\calM_{H>0}(p,q)$ depends only on the homotopy type of the ambient space $L(p,q)$. An illustrative example is given by $L(7,1)$ and $L(7,2)$. These two lens spaces have the same homotopy type (and therefore also isomorphic fundamental groups and the same homology). Nevertheless, by Theorem \ref{theo.mainone}, $\calM_{H>0}(7,1)$ is path-connected while $\calM_{H>0}(7,2)$ is not.
\end{rem}

By a result of Bonahon~\cite{Bo83}, any two Heegaard tori in a fixed lens space $L(p,q)$ are isotopic. In particular, when $q\not\cong \pm 1 \!\mod p$ we can pick two Heegaard tori $T_1$ and $T_2$ lying in different components of $\calM_{H>0}(p,q)$; then $T_1$ and $T_2$ are isotopic, but no isotopy between them preserves mean convexity. By contrast, our next result says that we can always find an isotopy that preserves nonnegativity of mean curvature (at least in the case where the ambient lens space has positive Ricci curvature).

\begin{theo}\label{theo.maintwo}
Let $p,q$ be coprime integers with $p > 0$. Fix a metric $g$ on $L(p,q)$ with $\Ric_g > 0$ (recall from above that this is always possible). Then $\calM_{H\geq 0}(p,q)$ is always path-connected.
\end{theo}

We finish this introduction with an outline of the proofs of the above theorems and of the article.  

First, in Section \ref{sec.glue}, we show that every mean convex $2$-sphere in a complete, orientable Riemannian three-manifold with nonnegative Ricci curvature as in Theorem \ref{theo.sphere} bounds a mean convex domain (a compact codimension $0$ submanifold with mean curvature vector pointing inwards). This important step allows us to work with domains instead of surfaces in the following; it also highlights that the theorem is sharp in the sense described in Remark \ref{rem.theo1sharp}. We then continue by extending a gluing result of the first author with Haslhofer and Hershkovits \cite{BHH21} from Euclidean space to the ambient manifold setting, allowing us to connect mean convex domains along thin ``strings'' in a way that preserves mean convexity. This gluing result could be of independent interest. 

Next, in Section \ref{sec.MCF}, we recall the theory of mean curvature flow with surgery from Haslhofer-Kleiner \cite{HK17} as extended by Haslhofer-Ketover \cite{HK19} to the ambient manifold setting. In particular, for any mean convex sphere or Heegaard torus $\Sigma$, there is a mean curvature flow (MCF) with surgery starting at $\Sigma$ that becomes extinct or converges to a union of minimal surfaces. If the ambient manifold has nonnegative Ricci curvature, then the only possibility is finite time extinction.

In Section \ref{sec.BHH}, we show that any element of $\calM_{H>0}(\Sph^2,M)$ can be deformed to a \emph{marble tree} which then itself can be deformed to a small geodesic sphere. Similarly, any element of $\calM_{H>0}(p,q)$ can be deformed to a \emph{marble circuit} which then itself can be deformed further to the boundary of the $\varepsilon$-neighbourhood of a simple closed curve $T_\eps=\partial N_\varepsilon(\gamma)$. These results follow again the ideas of the first author with Haslhofer and Hershkovits \cite{BHH21,BHH19}, using the gluing construction of Section \ref{sec.glue} to undo the surgeries. 

The proofs of the three theorems above are then completed in Section \ref{sec.topology}. Theorem \ref{theo.sphere} follows immediately from the results in Section \ref{sec.BHH}. Theorem \ref{theo.mainone} requires further topological arguments. First, we show that $\calM_{H>0}(p,q)$ has at most two components. This uses critically Bonahon's uniqueness theorem~\cite{Bo83}: since the initial torus is Heegaard, the simple closed curve $\gamma$ from the above paragraph must be isotopic (up to orientation) to one of the cores $\gamma_1$ or $\gamma_2$. In the case where $q\cong \pm 1 \!\mod p$, path-connectedness then follows immediately from a Lemma of Bonahon that says that in this case $\gamma_1$ and $\gamma_2$ are isotopic. The second case is by contradiction: let $T_1$ and $T_2$ be the boundary of the $\varepsilon$-neighbourhood of $\gamma_1$ and $\gamma_2$, respectively. Then $T_1$ (respectively $T_2$) has positive mean curvature and the mean curvature vector points towards $\gamma_1$ (respectively $\gamma_2$). If they were in the same component of the moduli space, along the isotopy the mean curvature would point in the same direction. The key point is that ``same'' direction is well-defined because $\gamma_1$ is not homologous to $\gamma_2$. Finally, to prove Theorem \ref{theo.maintwo}, we show that every Heegaard torus with $H\geq 0$ is isotopic (via an isotopy preserving $H \geq 0$) to a strictly mean convex Heegaard torus. Therefore $\calM_{H\geq 0}(p,q)$ cannot have more path-components than $\calM_{H>0}(p,q)$. Moreover, in the case where $\calM_{H>0}(p,q)$ has two path components there exists a minimal Heegaard torus $\Sigma$ by \cite{KMN} and we prove that it is isotopic \emph{on each side} to a mean convex torus. The result follows.
\vspace{-2mm}

\paragraph{Acknowledgements.} RB thanks Gianmichele Di Matteo and Luciano Mari for interesting discussions. He has been partially supported by the EPSRC grant EP/S012907/1. SM thanks Giuseppe Pipoli for interesting discussions. He has been partially supported by Agence Nationale de la Recherche grant ANR-17-CE40-0034. The authors also thank the anonymous referee for their helpful comments, in particular regarding Theorem \ref{theo.maintwo}.


\section{Mean convex gluing in ambient three-manifolds}\label{sec.glue}

Here and in the following two sections, $(M,g)$ denotes a complete, orientable Riemannian three-manifold with nonnegative Ricci curvature $\Ric_g \geq 0$ (unless mentioned otherwise). Since $M$ is orientable, all orientable embedded surfaces in $M$ are two-sided. This applies in particular to spheres and tori.

A \emph{domain} in $M$ is a (possibly disconnected) smooth \emph{compact} codimension 0 submanifold with boundary. A domain $K$ is \emph{mean convex} if its boundary $\Sigma=\bord K$ satisfies $H\neq 0$ everywhere, with mean curvature vector pointing toward $K$. 

In the following, instead of working with the mean convex surfaces $\Sigma$ in $(M,g)$, it will be convenient to work with the mean convex domains bounded by them. In the situations relevant for us, this can always be done. For mean convex Heegaard tori in lens spaces this is easy. Indeed, they bound two solid tori and the condition of mean convexity means that the mean curvature vector always points to one of these two solid tori which will therefore be our mean convex domain. For mean convex spheres as in Theorem \ref{theo.sphere} this is less obvious; we therefore start with a proof of this fact.

\begin{prop}\label{prop.bounddomain}
Let $\Sigma$ be a mean convex embedded two-sphere in $(M,g)$. Then $\Sigma$ bounds a mean convex domain $\bar{K}$.
\end{prop}

\begin{rem}\label{remark.bounddomain}
Note that the proposition is sharp in the sense that the condition $H>0$ cannot be weakened to $H \geq 0$, as seen by the counterexample $\Sigma = \Sph^2 \times \{0\}$ in $M=\Sph^2 \times \Sph^1$ (or $\Sph^2 \times \mathbb{R}$) with standard product metric. Similarly, the condition $\Ric_g \geq 0$ on $(M,g)$ cannot be weakened to $\Ric_g \geq -C$ for some $C>0$ (or replaced by $1 \geq \Ric_g \geq -\eps$ for some $\eps>0$). Also here the counterexample is $\Sigma = \Sph^2 \times \{0\}$ in $M=\Sph^2 \times \Sph^1$ (or $\Sph^2 \times \mathbb{R}$), this time with a warped product metric $g = w(x) g_{\Sph^2} + dx^2$ arbitrarily close to the standard product metric (i.e. with warping factor $w(x)$ close to $1$) and with $w'(0)\neq 0$ to guarantee that $\Sigma$ has $H>0$. Finally,  completeness of $M$ cannot be removed as seen by the counterexample $M=\Sph^3 \setminus \{p\}$ with $\Sigma$ a small geodesic two-sphere around $p$. The sphere $\Sigma$ bounds a compact domain only on one side, but the mean curvature vector points to the complement of this domain. All these counterexamples also work as counterexamples for path-connectedness of the moduli space $\calM_{H>0}(\Sph^2,M)$ in Theorem \ref{theo.sphere} if one of the assumptions is removed or weakened.
\end{rem}

\begin{proof}
Assuming towards a contradiction that there exists an embedded 2-sphere in $M$ with $H>0$ that does not bound a domain, we can lift it to the universal cover $\widetilde{M}$ (due to simply-connectedness). The lifted two-sphere is still two-sided, still has $H>0$, and still does not bound a domain. Hence we can work directly in the universal cover.

Since $M$ is complete, orientable, and has $\Ric_g \geq 0$, a theorem of Liu~\cite{Liu} implies that its universal cover $\widetilde{M}$ is either diffeomorphic to $\Sph^3$ or $\mathbb{R}^3$ or it splits isometrically as a Riemannian product $\widetilde{M} = \Sph^2 \times \mathbb{R}$, where the $\Sph^2$ factor has (possibly nonconstant) nonnegative curvature. We consider these cases separately.

\begin{case}
$\widetilde{M}$ is diffeomorphic to $\Sph^3$.
\end{case}

This is the easiest case. By Alexander's theorem~\cite{Alex}, $\Sigma$ bounds a 3-ball on either side. One of these 3-balls is a mean convex domain.

\begin{case}
$\widetilde{M}$ is diffeomorphic to $\mathbb{R}^3$.
\end{case}

Again, by Alexander's theorem~\cite{Alex}, $\Sigma$ bounds a 3-ball. We need to prove that the mean curvature vector points towards this ball. By Kasue's theorem \cite[Theorem C]{Kasue}, if $H \geq 0$ with respect to the \emph{outward pointing normal}, i.e. the mean curvature vector points to the complement of the ball, then this complement is isometric to a Riemannian product $\Sigma \times [0,\infty)$ and $\Sigma$ is totally geodesic, i.e. $H\equiv 0$. This shows that the mean curvature vector must point inwards. (See also Theorem 1.6. of \cite{AFM} for a strengthening of Kasue's theorem that yields a direct proof that the mean curvature vector points towards the ball in our case). As seen in the remark above, completeness of $\widetilde{M}$ is essential for this argument.

\begin{case}
$\widetilde{M}$ is a Riemannian product $\Sph^2 \times \mathbb{R}$ with nonnegatively curved $\Sph^2$ factor.
\end{case}

Let $\pi_2 : \widetilde{M} = \Sph^2 \times \mathbb{R} \to \mathbb{R}$ denote the projection to the second factor.

Assume $\Sigma$ does not bound a ball. Call $A^-$ and $A^+$ the two connected components of $(\Sph^2 \times \RR) \setminus \Sigma$. Both components are homeomorphic to $\Sph^2 \times \RR$ and adjacent to exactly one end. We assume that $p(A^+)$ (respectively $p(A^-)$) is adjacent to $+\infty$ (resp. $-\infty$). Since $H>0$, the mean curvature vector of $\Sigma$ always points to one of the two components. Without loss of generality, we may assume that it always points to $A^+$ (otherwise invert the orientation of $\RR$, swapping $A^-$, $A^+$).

Now consider a maximal point $p$ of the restricted projection $\pi_2 |_{\Sigma}$. On the one hand, by construction, the mean curvature vector of $\Sigma$ at $p$ points to $A^+$. On the other hand, by definition of $p$, $\Sigma$ lies on the ``lower'' side of the (minimal) two-sphere $\Sph^2 \times \{ \pi_2(p)\}$ and touches it at $p$. The maximum principle thus forces the mean curvature vector of $\Sigma$ at $p$ to either vanish or point to $A^-$, yielding the desired contradiction.

Hence $\Sigma$ bounds a ball $B$. Again looking at a maximum of the restriction of $\pi_2$ to $\Sigma$ and using the maximum principle we see that the mean curvature vector points toward $B$ at such a point. We conclude that $B$ is a mean convex domain.
\end{proof}

In what follows, we fix a compact domain $\bar{K}$ (which in our applications with noncompact $(M,g)$ will be the mean convex domain bounded by $\Sigma$) and only consider subsets of $\bar{K}$. In particular, everything below might depend on this choice of $\bar{K}$. By compactness, the curvature tensor and its derivatives are uniformly bounded over $\bar{K}$ and $\inj(\bar{K})>0$, where $\inj(\bar{K})$ denotes the infimum of the injectivity radius of $M$ at $p$ over all points $p\in \bar{K}$.

We denote by $\calD$ the set of all mean convex domains contained in $\bar{K}$. For such a domain $K\in \calD$, $p\in\bord K$, and $\alpha>0$, we denote by $\bar B^\pm_\alpha(p)$ the closed metric balls of radius $\alpha H(p)^{-1}$ with center $\exp_p(\pm\alpha H(p)^{-1}\nu(p))$ in $M$ where $\nu(p)$ is the unit normal vector to $\bord K$ at $p$ pointing towards $K$.

\begin{defi}[Quantitative noncollapsing, see \cite{SW09, An12, HK17, HK19}]\label{def.noncollapsed}
Let $\alpha>0$. We say that a mean convex domain $K \in \calD$ is \emph{$\alpha$-noncollapsed} if for every $p\in\bord K$ the following conditions hold:
\begin{enumerate}
\item $H(p)\ge 4\alpha \inj(\bar{K})^{-1}$;
\item  $\bar B^+_\alpha(p)\subset K$;
\item $\bar B^-_\alpha(p) \cap K \subset \bord K$.
\end{enumerate}  
\end{defi}

A \emph{curve} in $M$ is a (possibly disconnected) smooth compact $1$-submanifold with boundary. The set of curves contained in $\bar{K}$ is denoted by $\calC$. In this section, we will glue certain \emph{controlled} domains along certain \emph{controlled} curves in $\bar{K}$.

\begin{defi}[Controlled domains and curves, slight modification of \cite{BHH21}]\label{def.controlled}
Let $\bA=(\alpha,c)$ be a pair of positive numbers. We say that a domain in $\bar{K}$ is \emph{$\bA$-controlled} if it is mean convex with $H \geq c$, $\alpha$-noncollapsed, and satisfies $|A| + |\nabla A| \leq c^{-1}$. The set of $\bA$-controlled domains in $\bar{K}$ is denoted by $\calD_\bA$.

Let $\bb>0$. We say that a curve $\gamma$ in $\bar{K}$ is $\bb$-controlled if its curvature vector satisfies $|\kappa| \leq \bb^{-1}$ and $|\nabla \kappa| \leq \bb^{-2}$, each connected component of $\gamma$ has normal injectivity radius $\geq \bb/10$, and different connected components of $\gamma$ are at least $10\bb$ apart. Moreover, we require that the $\bb/10$-tubular neighbourhood of $\gamma$ is still contained in $\bar{K}$. We denote by $\calC_\bb$ the set of $\bb$-controlled curves.

An \emph{$(\bA,\bb)$-controlled configuration} is a pair $(K,\gamma) \in \calD_\bA \times \calC_\bb$ satisfying the following properties:
\begin{enumerate}
\item The interior of $\gamma$ is contained in the complement of $K$.
\item If $p\in \bord \gamma$, then $p\in \bord K$ and $\gamma$ meets $\bord K$ orthogonally at $p$.
\item $d(\gamma \setminus \bigcup_{p\in\bord\gamma} B_{\bb/10} (p), \bord K)\ge \bb/20$. 
\end{enumerate}
We denote the set of $(\bA,\bb)$-controlled configurations by $\calX_{\bA,\bb}$.
\end{defi}

We note that if $\bA'=(\alpha',c')$ is such that $\alpha'\leq\alpha$ and $c'\leq c$ then $\calD_\bA \subseteq \calD_{\bA'}$ and if $\bb'\leq \bb$ we have $\calC_\bb \subseteq \calC_{\bb'}$. In particular also $\calX_{\bA,\bb} \subseteq \calX_{\bA',\bb'}$.

The main result of this section is the following theorem which can be seen as an ambient manifold version of the gluing result from \cite[Theorem~4.1]{BHH21}.

\begin{theo}[Mean convex gluing]\label{thm.glue}
There exists a smooth gluing map
\begin{equation*}
\mathcal{G}:\calX_{\bA,\bb}\times (0,\bar{r})\to \calD,\quad ((K,\gamma),r)\mapsto \mathcal{G}_{r}(K,\gamma),
\end{equation*}
where $\bar{r}=\bar{r}(\bA,\bb,\bar{K},M,g)>0$ is a constant, and a smooth increasing function $\delta:(0,\bar{r})\rightarrow \mathbb{R}_+$ with $\lim_{r\rightarrow 0}\delta(r)=0$, with the following properties:
\begin{enumerate}
\item For every $(K,\gamma) \in \calX_{\bA,\bb}$ and every $0<r<\bar{r}$, the resulting domain $\mathcal{G}_{r}(K,\gamma)$ deformation retracts to $K\cup \gamma$.
\item The nontrivial part of the gluing only happens near the points $p\in\partial \gamma$ in the following sense:
\begin{equation*}
\mathcal{G}_{r}(K,\gamma)\setminus\bigcup_{p\in \bord \gamma} B_{2\delta(r)}(p) = K \cup  N_{r}(\gamma)\setminus \bigcup_{p\in \bord \gamma} B_{2\delta(r)}(p),
\end{equation*}
where $N_{r}(\gamma)$ denotes the solid $r$-tubular neighbourhood of $\gamma$.
\end{enumerate}
\end{theo}

We will often refer to the regions $N_{r}(\gamma)$ as \emph{strings} and call $r$ the \emph{string radius}.

\begin{rem}
For simplicity, we only state and prove the theorem for (uniformly) mean convex domains in ambient three-manifolds, but with some very minor modifications of the proof, one obtains the same result in all dimensions for (uniformly) two-convex domains. Moreover, the result in \cite{BHH21} also allows $\gamma$ to have ``loose end points'', that is, points $p\in \bord\gamma \setminus \bord K$, but we omit them here for simplicity. In fact, in \cite{BHH21}, these loose end points only come up in the ``marble reduction argument'', and while such an argument is also present in our proof of Theorem \ref{theo.isotopy2} below, we can simply pass to the Euclidean setting via the exponential map, as this argument can be performed at arbitrarily small scales where ambient curvature becomes negligible.
\end{rem}

The idea of the proof is as follows: If $p\in \bord\gamma \cap \bord K$, we lift the domain $K$ and curve $\gamma$ in a neighbourhood of $p$ to the tangent space via the exponential map, use the gluing construction from Theorem 4.1 in \cite{BHH21} in Euclidean $\RR^3$, and then use the exponential map again to obtain a domain in $M$. The resulting domain is not yet exactly what we want, as the ``string part'' will be the image (under the exponential map) of a tubular neighbourhood in Euclidean space rather than a tubular neighbourhood of $\gamma$ in $(M,g)$, thus requiring a further modification in the annular regions $B_{2\delta(r)}(p) \setminus B_{\delta(r)}(p)$. All of this is relatively easy if the mean curvature is very large (in which case the errors caused by the ambient curvature of $(M,g)$ will essentially be irrelevant), but it becomes more tricky when gluing at points of ``small'' (but positive) mean curvature. To deal with this problem carefully, we will need the following technical lemma.

Recall that for each $\varrho<\inj(M,p)$, the exponential map $\exp_p: T_pM \simeq \mathbb{R}^3 \to M$ is a diffeomorphism from the Euclidean $\varrho$-ball $B_\varrho(0) \subset \mathbb{R}^3$ to the geodesic $\varrho$-ball $B_\varrho(p) \subset M$ (where $0$ is mapped to $p$).\footnote{More precisely, after picking an oriented orthonormal basis $\{E_1,E_2,E_3\}$ of $T_p M$ at $p$, we identify $\mathbb{R}^3$ with $T_pM$ via the linear isometry $\iota:\mathbb{R}^3 \to T_pM$ given by $\iota(x^1,x^2,x^3)=\sum_{k=1}^3 x^kE_k$.}

\begin{lemma}\label{lemma.glue}
Let $p\in \bar{K}$ and $\varrho<\inj(\bar{K})$ be fixed. Let $S\subset B_\varrho(0)\subset \mathbb{R}^3$ be a two-sided surface and $\Sigma =\exp_p(S) \subset B_\varrho(p)\subset M$ its image under the above diffeomorphism and let these surfaces be endowed with normal vector fields pointing ``in the same direction''. We denote by $A^S$ and $H^S$ the second fundamental form and mean curvature of $S$ in $\mathbb{R}^3$ and by $A^\Sigma$ and $H^\Sigma$ the second fundamental form and mean curvature of $\Sigma$ in $(M,g)$ with respect to these normal vector fields. Then the following holds.
\begin{enumerate}
\item\label{lemma.p1} If $H^\Sigma \geq c_\Sigma>0$, $|A^\Sigma| + |\nabla A^\Sigma| \leq c_\Sigma^{-1}$, and $\Sigma$ is $\alpha$-noncollapsed everywhere in $B_\varrho(p)$, then there is $0<\delta_1\leq \varrho$ depending on the curvature of $(M,g)$ in $p$ and on $c_\Sigma$ and $\alpha$, such that $S \cap B_{\delta_1}(0)$ satisfies $H^S \geq \frac{9}{10}c_\Sigma$, $|A^S| + |\nabla A^S| \leq \frac{10}{9}c_\Sigma^{-1}$, and is $\frac{9}{10}\alpha$-noncollapsed.
\item\label{lemma.p2} If $\gamma \subset B_\varrho(p)$ is a $\bb$-controlled curve, then there is $0<\delta_2 \leq \varrho$ depending on the curvature of $(M,g)$ in $p$ and on $\bb$, such that $\exp_p^{-1}(\gamma) \cap B_{\delta_2}(0)$ is $\frac{9}{10}\bb$-controlled.
\item\label{lemma.p3} If $H^S \geq c_S>0$ everywhere in $B_\varrho(0)$, then there is $0<\delta_3\leq \varrho$ depending on the curvature of $(M,g)$ in $p$ and on $c_S$, such that $\Sigma \cap B_{\delta_3}(p)$ has $H^\Sigma \geq \frac{9}{10}c_S$.
\end{enumerate}
\end{lemma}

\begin{proof}[Proof of Theorem \ref{thm.glue} (assuming Lemma \ref{lemma.glue})]
As the construction is local, we first explain the gluing around one point $p\in \partial\gamma \cap \partial K$, making sure no constants will depend on the choice of this $p$ (but possibly on the set $\bar{K}$). First, set $\varrho=\frac{1}{2}\min\{\inj(\bar{K}),\min_{q_1,q_2\in\partial\gamma}d(q_1,q_2)\}$, guaranteeing that $\exp_p$ restricted to $B_\varrho(0) \subset \mathbb{R}^3$ is a diffeomorphism and $p$ is the only gluing point contained in $B_\varrho(p)$. We denote $\Sigma=\partial K \cap B_\varrho(p)$ and $S=\exp_p^{-1}(\Sigma) \subset B_\varrho(0)\subset\mathbb{R}^3$. 

By assumption, $K$ is $\bA=(\alpha,c)$-controlled and $\gamma$ is $\bb$-controlled. We can therefore let $\delta_1$ be as in Point \ref{lemma.p1} of Lemma \ref{lemma.glue} with $c_\Sigma=c$, and $\delta_2$ as in Point \ref{lemma.p2} of Lemma \ref{lemma.glue}. This ensures that for $\delta_0=\min\{\delta_1,\delta_2\}$ the surface $S\cap B_{\delta_0}(0)$ has $H^S \geq \frac{9}{10}c$, $|A^S| + |\nabla A^S| \leq \frac{10}{9}c^{-1}$ and is $\frac{9}{10}\alpha$-noncollapsed and moreover $\widehat\gamma = \exp_p^{-1}(\gamma) \cap B_{\delta_0}(0)$ is $\frac{9}{10}\bb$-controlled. We can therefore apply the Euclidean gluing result from Theorem 4.1 of \cite{BHH21} with these constants after extending $\exp_p^{-1}(K) \cap B_{\delta_0}(0)$ to some $\bA'=(\frac{9}{10}\alpha,\frac{9}{10}c)$-controlled domain $\widehat{K}$ in $\mathbb{R}^3$ that satisfies $\partial \widehat{K} \cap B_{\delta_0}(0) = S$. This yields \emph{for every sufficiently small $r<\bar{r}$} a mean convex domain $\mathcal{G}_{r}(\widehat K,\widehat\gamma)$ in $\mathbb{R}^3$ such that
\begin{equation*}
\mathcal{G}_{r}(\widehat K,\widehat \gamma)\setminus B_{\delta(r)}(0) = \widehat K \cup  N_{r}(\widehat \gamma)\setminus B_{\delta(r)}(0),\end{equation*}
where $\delta(r)$ and $\bar{r}$ are given by Theorem 4.1 in \cite{BHH21}. Moreover, by possibly further decreasing $\bar r$ we may also assume that $\delta(r) < \delta_0$. Tracing through the proof, we see that in fact $\mathcal{G}_{r}(\widehat K,\widehat\gamma)$ is not only mean convex, but has an explicit lower bound $\widehat c(r)>0$ for the mean curvature which only depends on $(\bA',\bb')$ (or equivalently on $(\bA,\bb)$) and on $r$ and does \emph{not} degenerate as $r\to 0$, hence $\inf_{r\in(0,\bar r)} \widehat c(r) =: \widehat{c}>0$. We will also make sure that $\bar r$ is small enough so that $\delta(r) < \frac{1}{4}\delta_3$, where $\delta_3>0$ is from Point \ref{lemma.p3} of Lemma \ref{lemma.glue} with $c_S=\widehat{c}$.

The region created this way is not yet exactly what we want, but the last adjustment is rather easy. For sufficiently small $\bar{r}$ (depending only on $\bb$), both $N_{r}(\widehat \gamma)$ and $\exp_p^{-1}(N_r(\gamma))$ have very large mean curvature, say at least $100\widehat{c}$. We can therefore easily interpolate from one region to the other in $B_{2\delta(r)}(0) \setminus B_{\delta(r)}(0)$ with the resulting region agreeing with $N_{r}(\widehat \gamma)$ on $B_{\frac{5}{4}\delta(r)}(0) \setminus B_{\delta(r)}(0)$ and with $\exp_p^{-1}(N_r(\gamma))$ on $B_{2\delta(r)}(0) \setminus B_{\frac{7}{4}\delta(r)}(0)$ and still having $\widehat{c}$ as lower bound for the mean curvature everywhere.

Applying $\exp_p$ in $B_{\delta_3}(0)$, by Point \ref{lemma.p3} of Lemma \ref{lemma.glue}, we obtain the desired $\mathcal{G}_{r}(K,\gamma) \cap B_{\delta_3}(p)$, which will have $\frac{9}{10}\widehat c$ as a lower bound for its mean curvature and agree with $K \cup N_r(\gamma)$ outside $B_{2\delta(r)}(p)$. Finally, we note that if $\bar{r}$ is small enough so that $\bar{r}, \delta(\bar{r}) < \bb/10$, then the resulting domain is still contained in $\bar{K}$. This finishes the proof.
\end{proof}

It remains to prove the technical lemma. The idea is that \emph{at} $p$ the second fundamental form and mean curvatures $A^S$ and $H^S$ of $S$ in $\mathbb{R}^3$ and $A^\Sigma$ and $H^\Sigma$ of $\Sigma$ in $(M,g)$ agree and from there the difference between the two values will grow by a function of the distance to $p$ -- similarly for the curvature of curves. Hence restricting to sufficiently small balls, the errors will be small enough to conclude the claims of the lemma.

\begin{proof}[Proof of Lemma \ref{lemma.glue}]
Let us first focus on the Points \ref{lemma.p1} and \ref{lemma.p3} concerned with the second fundamental form and mean curvature of the two surfaces.

We use the diffeomorphism $\exp_p: \mathbb{R}^3 \supset B_\varrho(0) \to B_\varrho(p) \subset M$ (after identifying $T_pM$ with $\mathbb{R}^3$ via a linear isometry) to pull back the metric $g$ on $M$ to $\mathbb{R}^3$, denoting the resulting metric with $\widetilde{g} = \exp_p^*g$. This way, $\exp_p:(B_\varrho(0), \widetilde{g}) \to (B_\varrho(p),g)$ is a Riemannian isometry.

We can then phrase the lemma in terms of the two different metrics on $B_\varrho(0)\subset\mathbb{R}^3$, proving that in sufficiently small balls all quantities we are interested in remain comparable in these different metrics. In particular, for a two-sided surface $S\subset B_r(0)\subset \mathbb{R}^3$, we denote by $A=h_{\alpha\beta}$ and $H$ its second fundamental form and mean curvature in $\mathbb{R}^3$ with the Euclidean metric $g_e$, and by $\widetilde{A}=\widetilde{h}_{\alpha\beta}$ and $\widetilde{H}$ its second fundamental form and mean curvature in $\mathbb{R}^3$ with respect to the metric $\widetilde{g}$. We then want to prove the following.
\begin{claim}\label{claim1}
$\widetilde{h}_{\alpha\beta}(q) = h_{\alpha\beta}(q) + O_p(\delta^2)$ and $\widetilde{H}(q)=H(q)+O_p(\delta^2)$ for all $q \in S \cap B_{\delta}(0)$, where the errors $O_p(\delta^2)$ depend on the curvature of $M$ and all its derivatives at the point $p\in \bar{K}$. 
\end{claim}
From this claim, the statement about the mean curvature in Point \ref{lemma.p1} of the lemma follows immediately by picking $\delta_1$ sufficiently small so that $O_p(\delta_1^2) \leq \tfrac{1}{10}c_\Sigma$, while Point \ref{lemma.p3} follows by picking $\delta_3$ small enough so that $O_p(\delta_3^2) \geq -\tfrac{1}{10}c_S$. Note that due to the compactness of $\bar{K}$, this can be done independently of the choice of $p$.

\begin{proof}[Proof of Claim \ref{claim1}]
It is well known that our choice of (normal) coordinates implies that for all $x\in B_\varrho(0)\subset\mathbb{R}^3$, we have
\begin{equation}\label{eq.normalexp}
\widetilde{g}_{ij}(x)= \delta_{ij}-\tfrac{1}{3} R_{i k j \ell}\; x^k x^{\ell}+ O(\abs{x}^3),
\end{equation}
see for example \cite{Wil96}. We note that the expression $O(\abs{x}^3)$ depends on the curvature in $p$ and all its derivatives. In the following, we normally drop dependence on $x\in \mathbb{R}^3$ where obvious to keep notation shorter and we also do not always explicitly keep track of the quadratic terms. 

Parametrise $S$ by $\theta=(\theta^1,\theta^2,\theta^3): U \to \mathbb{R}^3$, where $U\subset \mathbb{R}^2$. Using coordinates $z^1,z^2$ in $U\subset\mathbb{R}^2$, we have $\theta^i=\theta^i(z^1,z^2)$, for $i=1,2,3$. We denote by $\theta_\alpha = (\theta^1_\alpha, \theta^2_\alpha, \theta^3_\alpha)$ the tangent vector $\frac{\partial}{\partial z^\alpha} \theta$, $\alpha=1,2$. In particular, we have the first fundamental form of $S$ in $(\mathbb{R}^3,g_e)$ given by
\begin{equation*}
g^S_{\alpha\beta}=g_e(\theta_\alpha,\theta_\beta)= \delta_{ij}\theta_\alpha^i \theta_\beta^j,
\end{equation*}
where here and in the following calculations indices such as $i, j, k, \ell$ are in $\{1,2,3\}$ and Greek indices like $\alpha,\beta$ are in $\{1,2\}$ and we use the Einstein summation convention. Working instead in the metric $\widetilde{g}$ on $\mathbb{R}^3$, we find the first fundamental form of $S$ at a point $x$ in $(\mathbb{R}^3,\widetilde{g})$ using \eqref{eq.normalexp},
\begin{equation*}
\widetilde{g}^S_{\alpha\beta}=\widetilde{g}(\theta_\alpha,\theta_\beta)= \widetilde{g}_{ij}\theta_\alpha^i \theta_\beta^j = g^S_{\alpha\beta} -\tfrac{1}{3} R_{i k j \ell}\; x^k x^{\ell} \theta_\alpha^i \theta_\beta^j + O(\abs{x}^3).
\end{equation*}
We point out that here and in the following, the terms $O(\abs{x}^k)$ might depend on the vectors $\theta_1$, $\theta_2$ up to at most quadratic order. The inverses also satisfy a similar relation. For us, it is not necessary to have an explicit expression for the terms that are quadratic in $x$, but it is sufficient to note that
\begin{equation}\label{eq.inversemetric}
(\widetilde{g}^S)^{\alpha\beta}=(g^S)^{\alpha\beta} + O(\abs{x}^2).
\end{equation}
The unit normal vector to $S$ for the Euclidean background metric is given by $\nu=(\theta_1 \wedge \theta_2) / \norm{\theta_1 \wedge \theta_2}$. We note that for sufficiently small $\varrho$, the unit normal vector $\widetilde{\nu}$ with respect to the background metric $\widetilde{g}$ lies very close to $\nu$, and we can therefore make the ansatz $\widetilde{\mu}=\nu+a^\alpha \theta_\alpha$ to obtain a (not yet normalised) normal vector. Indeed, the condition that $\widetilde{\mu}$ is normal to $\theta_\beta$ with respect to the metric $\widetilde{g}$ yields
\begin{equation*}
0=\widetilde{g}(\widetilde{\mu},\theta_\beta) = \widetilde{g}(\nu,\theta_\beta) + a^\alpha \widetilde{g}(\theta_\alpha,\theta_\beta) = -\tfrac{1}{3} R_{i k j \ell}\; x^k x^{\ell} \nu^i \theta^j_\beta + O(\abs{x}^3) + a^\alpha \widetilde{g}^S_{\alpha\beta},
\end{equation*}
where in the last step we used the expansion \eqref{eq.normalexp} as well as $g_e(\nu,\theta_\beta)=0$. Solving for $a^\alpha$, we find
\begin{equation*}
a^\alpha = (\widetilde{g}^S)^{\alpha\beta}\Big(\tfrac{1}{3} R_{i k j \ell}\; x^k x^{\ell} \nu^i \theta^j_\beta + O(\abs{x}^3)\Big) =  (g^S)^{\alpha\beta} \Big(\tfrac{1}{3} R_{i k j \ell}\; x^k x^{\ell} \nu^i \theta^j_\beta + O(\abs{x}^3)\Big).
\end{equation*}
In order to normalise $\widetilde{\mu}$, we note that
\begin{align*}
\norm{\widetilde{\mu}}_{\widetilde{g}}^2 &= \widetilde{g}(\widetilde{\mu},\widetilde{\mu}) = \widetilde{g}(\widetilde{\mu},\nu+a^\beta\theta_\beta) = \widetilde{g}(\widetilde{\mu},\nu) =\widetilde{g}(\nu,\nu)+a^\beta \widetilde{g}(\theta_\beta,\nu)\\
&=\norm{\nu}^2_{g_e} - \tfrac{1}{3} R_{i k j \ell}\; x^k x^{\ell} \nu^i \nu^j + O(\abs{x}^3) + a^\beta\big(g_e(\theta_\beta,\nu) -\tfrac{1}{3} R_{i k j \ell}\; x^k x^{\ell} \theta_\beta^i \nu^j + O(\abs{x}^3)\Big)\\
&= 1 - \tfrac{1}{3} R_{i k j \ell}\; x^k x^{\ell} \nu^i \nu^j + O(\abs{x}^3).
\end{align*}
On the first line of this calculation, we used that $\widetilde{g}(\widetilde{\mu},\theta_\beta)=0$, for the second line we used the expansion \eqref{eq.normalexp}, and finally in the very last step we used $\norm{\nu}^2_{g_e}=1$, $g_e(\theta_\beta,\nu)=0$, and the fact that $a^\alpha$ grows quadratically in $x$. Therefore 
\begin{equation}\label{eq.inverseM}
\norm{\widetilde{\mu}}_{\widetilde{g}}^{-1} = 1 + \tfrac{1}{6} R_{i k j \ell}\; x^k x^{\ell} \nu^i \nu^j + O(\abs{x}^3)
\end{equation} 
and we find for the unit normal vector 
\begin{equation*}
\widetilde{\nu} =\widetilde{\mu} / \norm{\widetilde{\mu}}_{\widetilde{g}} = \Big(\nu + \tfrac{1}{3} (g^S)^{\alpha\beta} R_{i k j \ell}\; x^k x^{\ell} \nu^i \theta^j_\beta\;\theta_\alpha + O(\abs{x}^3)\Big)\Big(1 + \tfrac{1}{6} R_{i k j \ell}\; x^k x^{\ell} \nu^i \nu^j + O(\abs{x}^3)\Big),
\end{equation*}
in particular $\widetilde{\nu}=\nu+O(\abs{x}^2)$.
We then calculate the second fundamental form and the mean curvature of $S$ with respect to the two different metrics. For the Euclidean metric, we have
\begin{equation*}
h_{\alpha\beta} = -g_e(\nabla_{\theta_\alpha} \nu,\theta_\beta), \quad H=(g^S)^{\alpha\beta}h_{\alpha\beta},
\end{equation*}
and similarly for the background metric $\widetilde{g}$, we have
\begin{equation*}
\widetilde{h}_{\alpha\beta} = -\widetilde{g}(\nabla_{\theta_\alpha} \widetilde{\nu},\theta_\beta), \quad \widetilde{H}=(\widetilde{g}^S)^{\alpha\beta}\widetilde{h}_{\alpha\beta}.
\end{equation*}
Multiplying with $\norm{\widetilde{\mu}}_{\widetilde{g}}$, we can calculate the second fundamental form in the latter case as
\begin{equation*}
\norm{\widetilde{\mu}}_{\widetilde{g}} \widetilde{h}_{\alpha\beta} = -\widetilde{g}(\nabla_{\theta_\alpha} \widetilde{\mu},\theta_\beta) = -\widetilde{g}(\nabla_{\theta_\alpha} \nu,\theta_\beta) - a^\gamma \widetilde{g}(\nabla_{\theta_\alpha} \theta_\gamma, \theta_\beta)
\end{equation*}
The first of the terms on the right hand side will give us $h_{\alpha\beta}$ as well as error terms at least quadratic in $x$. The second term on the other hand will always be at least quadratic in $x$ because of the factor $a^\gamma$. We can therefore arrive at
\begin{equation*}
\norm{\widetilde{\mu}}_{\widetilde{g}} \widetilde{h}_{\alpha\beta} = h_{\alpha\beta} + O(\abs{x}^2).
\end{equation*}
Proving this precisely requires a rather lengthy calculation (involving also estimates for the Christoffel symbols with respect to the different background metrics), and in an attempt to keep this article short and because the calculation does not offer any useful further insight, we skip the details here. We refer the reader to \cite{DiMM} where such a calculation has been carried out in full detail in the more specific situation where $S$ is a perturbation of a sphere.

Now using first our formula \eqref{eq.inverseM} we obtain $\widetilde{h}_{\alpha\beta} = h_{\alpha\beta} + O(\abs{x}^2)$ and then using \eqref{eq.inversemetric}, we deduce that $\widetilde{H}=H+O(\abs{x}^2)$, concluding the claim.
\end{proof}

The statement about the second fundamental form in Point \ref{lemma.p1} of the lemma follows by first showing that the first derivatives of the second fundamental forms satisfy a similar relation as the second fundamental forms themselves. This again involves estimating Christoffel symbols and we leave the details to the reader. The claimed estimates involving norms of the second fundamental form and its derivative then follow from \eqref{eq.inversemetric}. 

We also note that $S$ is $\alpha$-noncollapsed (with respect to the pulled back metric $\widetilde{g}$) if for each point $x \in S$ the balls with radius $\widetilde{r}(x)=\alpha/\widetilde{H}(x)$ and center $\widetilde{C}= x \pm \widetilde{r}(x)\widetilde{\nu}(x)$ touch $S$ only in the point $x$. By Claim \ref{claim1}, for $\delta$ sufficiently small, the mean curvature $H(x)$ of $S$ at $x$ with respect to the Euclidean metric will be extremely close to $\widetilde{H}(x)$. Moreover, as seen in the proof of Claim \ref{claim1}, the unit normal vectors $\nu(x)$ and $\widetilde{\nu}(x)$ will be very close as well. Therefore, for sufficiently small $\delta$, the above balls contain the balls with radius $r(x)=\tfrac{9}{10} \alpha/H(x)$ and center $C = x \pm r(x)\nu(x)$. If follows that with respect to the Euclidean metric, $S$ is $\tfrac{9}{10}\alpha$-noncollapsed, finishing the proof of Point \ref{lemma.p1} of the lemma.

Finally, the remaining Point \ref{lemma.p2} of the lemma follows in a similar way as the other two points and we leave the direct but lengthy calculation to the reader.
\end{proof}


\section{Mean curvature flow with surgery}\label{sec.MCF}

In this section we provide a short overview of the main definitions and results from the theory of mean curvature flow with surgery, following the terminology of Haslhofer-Kleiner \cite{HK17} (in Euclidean space) and its extension by Haslhofer-Ketover \cite{HK19} to the ambient three-manifold setting. We also refer the interested reader to slightly earlier work of Brendle-Huisken \cite{BH18}, where a similar theory of mean curvature flow with surgery in three-manifolds has been developed, extending their prior results from \cite{BH16} in Euclidean $\RR^3$ (which itself extended results of Huisken-Sinestrari \cite{HS09}).

As we have seen at the beginning of Section \ref{sec.glue}, in the situations we are interested in, rather than working with mean convex surfaces we can work with the mean convex domains bounded by them. A smooth family of mean convex domains $\{K_t\}_{t\in I}$ in a complete three-manifold $(M,g)$ is said to move by \emph{mean curvature flow} if every point in the boundary $\bord K_t$ evolves with time derivative equal to the mean curvature vector of the boundary. It is a well known fact that mean convexity is preserved along the mean curvature flow, see also \eqref{eq.evolH} below.

The reason to introduce a surgery procedure, in a nutshell, is that typically mean curvature flow develops singularities where it can no longer be smoothly extended in time. In the general case, singularities can be very complicated (for instance, singularity models can have arbitrary genus, see \cite{BNS21}), but in the mean convex case all local singularities have a (possibly degenerate) cylindrical neck structure. Hence, to avoid dealing with the singularities, one stops the flow shortly before the singularity occurs, replaces neck regions by pairs of opposing caps, discards regions of high curvature, and then restarts the flow. This surgery procedure allows to extend the flow until it becomes extinct. In the following, we will recall the rigorous definitions and existence results from \cite{HK17, HK19}.

\begin{rem}\label{rem.compact}
Haslhofer-Ketover \cite{HK19} work in closed Riemannian manifolds, while our $(M,g)$ is only complete. This however does not create any additional complications because our evolving domains $K_t$ are all contained in the compact domain $\bar{K}$, hence we could change $(M,g)$ into a closed manifold by modifying it outside $\bar{K}$. 
\end{rem}

We want to define a mean curvature flow with surgery starting from an element $K_0$ in $\calD$ (allowing in particular also $K_0=\bar{K}$), but for simplicity, following \cite{HK17, HK19}, we first define a more general class of flows called $(\alpha,\delta)$-flows.

\begin{defi}[$(\alpha,\delta)$-flows, see \cite{HK17, HK19}]\label{def.alphadelta}
An \emph{$(\alpha,\delta)$-flow} is a collection of smooth families of mean convex domains $\{K^i_t\}_{t\in [t_{i-1},t_i]}$ ($i=1,\ldots,k$ and $0=t_0 < t_1 < \ldots < t_k < \infty$), such that:
\begin{enumerate}
\item For each $t$, the domain $K_t$ is $\alpha$-noncollapsed as in Definition \ref{def.noncollapsed}\footnote{While $\alpha$-noncollapsedness is preserved in Euclidean space, this is not the case in ambient manifolds; but by an estimate of Brendle \cite{Br16} the noncollapsing parameter can decay at most exponentially and therefore one could also assume this condition only for the initial domain.}
\item For each $i$, the final time slices of some (possibly empty) collection of disjoint strong $\delta$-necks are replaced by pairs of opposing  standard caps as described in Definition \ref{def.replacement} below, yielding a post-surgery domain $K^\sharp_{t_i}\subseteq K^i_{t_i} =: K^{-}_{t_i}$.\label{item2.alphadelta}
\item For each $i$, the initial time slice of the next flow $K^{i+1}_{t_i} =: K^{+}_{t_i}$ is obtained from the post-surgery domain $K^\sharp_{t_i}$ by discarding some (possibly empty) collection of connected components.\label{item3.alphadelta}
\item All surgeries are at comparable scales, meaning there exists a radius $s_0>0$, such that all $\delta$-necks in Point \ref{item2.alphadelta} of this definition have radius $s \in [\tfrac{1}{2}s_0,2s_0]$ for each $i$.
\end{enumerate}
\end{defi}

Let us discuss the second point of this definition in more detail. 

An $(\alpha,\delta)$-flow has a \emph{$\delta$-neck} with center $p$ and radius $s$ at time $t_0$ if $4s/\delta \leq \inj(\bar{K})$ as well as $s^{-1} \exp_p^{-1}(K_{t_0} \cap B_{2s/\delta}(p))$ is $\delta$-close in $C^{\lfloor{1/\delta}\rfloor}$ in $B_{1/\delta} \cap s^{-1} \exp_p^{-1}(B_{2s/\delta}(p)) \subset \RR^3$ to a solid round cylinder $D^2 \times \RR$ with radius $1$. Moreover, it has a \emph{strong} $\delta$-neck with center $p$ and radius $s$ at time $t_0$ if $4s/\delta \leq \inj(\bar{K})$ and $\{s^{-1} \exp_p^{-1}(K_{t_0+s^2t} \cap B_{2s/\delta}(p))\}_{t\in(-1,0]}$ is $\delta$-close in $C^{\lfloor{1/\delta}\rfloor}$ in $B_{1/\delta} \cap s^{-1} \exp_p^{-1}(B_{2s/\delta}(p)) \subset \RR^3$ to the evolution of a solid round cylinder $D^2 \times \RR$ with radius $1$ at $t = 0$. Finally, a \emph{standard cap} is a smooth, closed, convex, connected, and unbounded codimension 0 submanifold with boundary $K_{st} \subset \RR^3$ that coincides with a solid round half-cylinder of radius $1$ outside a ball of radius $10$. We can create a rotationally symmetric standard cap by smoothing out the set obtained from a half-ball attached to a half-cylinder. In the following, we assume that some standard cap has been fixed.

We can now define what we mean by replacing a strong $\delta$-neck with two opposing standard caps in the above definition.

\begin{defi}[Surgery procedure, see \cite{HK17, HK19}]\label{def.replacement}
After fixing a large cap-separation parameter $\Gamma$, we say that the final time slice of a strong $\delta$-neck ($\delta \leq 1/10\Gamma$) with center $p$ and radius $s$ is \emph{replaced by a pair of standard caps}, if (in a neighbourhood of $p$) the pre-surgery domain $K^{-}$ is replaced by a post-surgery domain $K^\sharp \subseteq K^{-}$ such that the following properties hold.
\begin{enumerate}
\item The replacement takes place inside $B=B_{5\Gamma s}(p)$.
\item For each integer $\ell \geq 0$ there exists a constant $C_\ell$, such that $\sup_{\bord K^\sharp \cap B} \lvert \nabla^\ell A \rvert \leq C_\ell s^{-(1+\ell)}$.
\item For each $p^\sharp \in K^\sharp$ with negative lower principal curvature $\lambda_1(p^\sharp) <0$ there is a corresponding $p^{-} \in K^{-}$ with $\lambda_1/H (p^\sharp) \geq \lambda_1/H (p^{-}) - o(s)$.
\item $s^{-1} \exp_p^{-1}(K^\sharp)$ is $o(\delta)$-close in $B_{10\Gamma}(0) \subset \RR^3$ to a pair of opposing standard caps at distance $\Gamma$ from the origin. Moreover there is an ``almost straight line'' between the tips of the caps, i.e.~a curve that is $o(\delta)$-close to a straight line in $\mathbb{R}^3$ and satisfies the estimates on $\kappa$ and $\nabla\kappa$ of an $o(\delta)^{-1}$-controlled curve.\label{item4.replacement}
\end{enumerate}
\end{defi}

We are finally ready to give the main definition of mean curvature flow with surgery, refining the definition of $(\alpha,\delta)$-flow.

\begin{defi}[MCF with surgery, see \cite{HK17, HK19}]
Fix a large cap-separation parameter $\Gamma$ and let $\bH=(H_\trig , H_\neck, H_\tth)$ be a triple of positive numbers with $H_\trig > H_\neck > H_\tth $, referred to as trigger, neck, and thick curvature scales. An $(\alpha,\delta,\bH)$-flow (also called a mean curvature flow with surgery) is an $(\alpha,\delta)$-flow with the following properties.
\begin{enumerate}
\item $H \leq H_\trig$ everywhere, and surgery/discarding occurs precisely at times when $H = H_\trig$ somewhere.
\item The collection of necks in Point \ref{item2.alphadelta} of Definition \ref{def.alphadelta} is a minimal collection of solid strong $\delta$-necks of curvature $H_\neck$ which separate $\{H = H_\trig\}$ from $\{H \leq H_\tth\}$ in $K^{-}_t$.
\item The collection of connected components that are discarded in Point \ref{item3.alphadelta} of Definition \ref{def.alphadelta} contains precisely the components with $H > H_\tth$ everywhere.\footnote{In particular, this means that of each
pair of opposing caps precisely one is discarded.}
\item Replacements are done in an optimal way in the sense that if a strong $\delta$-neck in Point \ref{item2.alphadelta} is in fact a strong $\widehat{\delta}$-neck for some $\widehat{\delta} < \delta$, then Point \ref{item4.replacement} of Definition \ref{def.replacement} also holds with this $\widehat{\delta}$.
\end{enumerate}
\end{defi}

\begin{theo}[Existence Theorem, Theorem 7.7 in \cite{HK19}]\label{thm.existence}
Let $K_0 \in \calD$ be a mean convex domain in $\bar{K}$ (allowing also $K_0=\bar{K}$). Then for every $T < \infty$, choosing $\alpha$ and $\delta$ small enough as well as scales $\bH$ large enough\footnote{To be precise, also the ratios $H_\trig/H_\neck$ and $H_\neck/H_\tth$ need to be sufficiently large.}, there exists an $(\alpha,\delta,\bH)$-flow $\{K_t\}_{t\in[0,T]}$ with initial condition $K_0$.
\end{theo}

An important point is that $\bar{K}$ is compact, thus the theory developed for closed manifolds also works, despite the fact that our $(M,g)$ is only complete, as explained in Remark \ref{rem.compact}. Generally, some components of a mean curvature flow with surgery might exist for all time and converge to a minimal surface in the ambient three-manifold, but in our specific case where the ambient manifold $M$ has nonnegative Ricci curvature $\Ric_g\geq 0$, we always obtain a finite extinction result.

\begin{prop}[Finite extinction]\label{prop.extenction}
Let $K_0 \in \calD$ be a mean convex domain in $\bar{K}$ satisfying $\min_{p\in\bord K_0} H(p)\geq c>0$. Then every mean curvature flow with surgery starting from $K_0$ as in Theorem \ref{thm.existence} becomes extinct in time $T_\ext < c^{-2}$.
\end{prop}

\begin{proof}
The evolution equation for the mean curvature (of $\bord K_t$) in the ambient manifold $(M,g)$ with $\Ric_g\geq 0$ yields
\begin{equation*}
\partial_t H = \Delta H + |A|^2 H + \Ric_g(\nu,\nu) H \geq \Delta H + \tfrac{1}{2} H^3
\end{equation*}
and hence in particular
\begin{equation}\label{eq.evolH}
\partial_t H_{\min}(t) \geq \tfrac{1}{2}H_{\min}^3(t),
\end{equation}
where $H_{\min}(t) := \min_{p\in\bord K_t} H(p)$. The surgery procedure does not influence this bound until the entire flow becomes extinct. Therefore, if the initial domain $K_0$ satisfies $H_{\min}(0) \geq c>0$, then by the comparison principle $H_{\min}(t) \geq (c^{-2}-t)^{-1/2} \nearrow \infty$ as $t\to c^{-2}$ and in particular the flow must become extinct in time $T_\ext < c^{-2}$ (which will also depend on the choice of scales $\bH$). Again, the comparison principle can be applied directly since we only work in the compact subset $\bar{K}$ of $(M,g)$.
\end{proof}

While the surgery procedure cuts our domain into several pieces, the gluing construction from Section \ref{sec.glue} can be used to connect it again with a very thin tubular string, using in particular also that all domains and curves are contained in $\bar{K}$. We glue in a string along the image $\gamma$ under the exponential map of the almost straight line between two opposing standard caps in a post-surgery region (as in Point \ref{item4.replacement} of Definition \ref{def.replacement}). We call $\gamma$ an \emph{almost geodesic}. An important result is that there is an isotopy between the pre-surgery domain and the domain obtained by first performing surgery and then gluing together again. Note that this latter domain in contained in the pre-surgery domain and hence in particular also in $\bar{K}$. It is important to preserve mean convexity, so we introduce the concept of \emph{mean convex isotopy} which is simply a smooth family $\{K_t\}$ of mean convex domains indexed by a real interval. We then have the following result that generalises \cite[Proposition 6.6]{BHH21} from the Euclidean case to the ambient three-manifold setting.

\begin{prop}[Isotopy from necks to glued caps]\label{prop.surgeryiso}
Assume that $K^\sharp$ is obtained from $K^{-}$ by replacing a strong $\delta$-neck by a pair of standard caps and assume that $\gamma$ denotes the almost geodesic connecting the tips of the caps in $M$. Then for $r$ small enough there is a mean convex isotopy between $K^{-}$ and $\mathcal{G}_{r}(K^\sharp,\gamma)$.
\end{prop}

\begin{proof}
Denote by $p_{\pm}$ the two gluing points and call their distance $d$. Set $\widetilde{K}_{\pm} = \exp_{p_{\pm}}^{-1}(K^{-})$ and $\widetilde{G}^1_{\pm} = \exp_{p_{\pm}}^{-1}(\mathcal{G}_{r}(K^\sharp,\gamma))$. If we work with $H_\neck$ sufficiently large and therefore the surgery happens in a sufficiently small ball around $p_{\pm}$, these regions are mean convex in $\mathbb{R}^3$ by Lemma \ref{lemma.glue}. As in the proof of Theorem \ref{thm.glue} we can easily interpolate between $\widetilde{G}_{\pm}^1$ and the ``Euclidean gluing'' $\widetilde{G}_{\pm}^2=\mathcal{G}_{r}(\exp_{p_{\pm}}^{-1}(K^\sharp),\exp_{p_{\pm}}^{-1}(\gamma))$ inside balls of radius $\tfrac{2}{3}d$ around the gluing points, because $\exp_{p_{\pm}}^{-1} (N_{r}(\gamma))$ and $N_{r}(\exp_{p_{\pm}}^{-1}(\gamma))$ have very large mean curvature (and the domains otherwise agree by construction). This gives a mean convex isotopy between $\widetilde{G}_{\pm}^1$ and $\widetilde{G}_{\pm}^2$ in balls of radius $\tfrac{2}{3}d$ around $p_{\pm}$ and brings us back to the Euclidean case where according to \cite[Proposition 6.6]{BHH21} there also exists a mean convex isotopy between $\widetilde{K}_{\pm}$ and $\widetilde{G}_{\pm}^2$. Combining the two mean convex isotopies, applying Lemma \ref{lemma.glue} one more time to go back to $M$, and interpolating in the regions where the balls overlap proves the proposition.
\end{proof}

The final ingredient that is needed is a description of the possible types of discarded components. For the following definition, we need the notion of a standard cap $K_{st}$ given above, as well as of an $(C,\eps)$-cap, by which we mean a closed, strictly convex, connected, and unbounded codimension 0 submanifold with boundary $K_{C,\eps} \subset \RR^3$ such that every point outside some compact subset of $\Int K$ with diameter at most $C$ is the center of some $\eps$-neck of radius $1$.

\begin{defi}[Capped tubes and tubular loops]\label{def.discarded}
A \emph{capped $\eps$-tube} is a mean convex domain $K$ diffeomorphic to $\bar{B}^3$ such that there is a curve $\gamma\subset K$ with $\bord \gamma = \gamma \cap \bord K=\{\bar p_+,\bar p_-\}$ and such that:
\begin{enumerate}
\item $H(\bar p_\pm)^{-1} \exp^{-1}_{\bar p_\pm} (K\cap B_{2CH(\bar p_\pm)^{-1}} (\bar p_\pm))$ is $\eps$-close, after rescaling, to either a $(C,\eps)$-cap $K_{C,\eps}$ or a standard cap $K_{st}$.
\item For every $p\in \gamma$, if $d(p,\bar p_+) \ge C H(\bar p_+)^{-1}$ and $d(p,\bar p_-) \ge C H(\bar p_-)^{-1}$ then $p$ is the center of some $\eps$-neck of some radius $s$ with axis tangent to $\gamma$, and such that $\gamma$ is $\eps^{-2}s$-controlled in $B_{\eps^{-1}s} (p)$.
\end{enumerate}
An \emph{$\eps$-tubular loop} is a mean convex solid torus admitting a core $\gamma$ with the following property: every point $p\in\gamma$ is center of an $\eps$-neck of some radius $s$ with axis tangent to $\gamma$, and such that $\gamma$ is $\eps^{-2}s$-controlled in $B_{\eps^{-1}s} (p)$.
\end{defi}

A direct consequence of the canonical neighbourhood theorems (Theorem 1.22 in \cite{HK17} and Theorem 7.6 in \cite{HK19}) and Corollary 1.25 in \cite{HK17} is the following result (see also \cite{BHH21, BHH19} for the corresponding statement in Euclidean space).

\begin{prop}[Discarded components]\label{prop.discarded}
All discarded components of the $(\alpha,\delta,\bH)$-flow $\{K_t\}$ from Theorem \ref{thm.existence} are diffeomorphic to $\bar{B^3}$ or to $D^2 \times \Sph^1$. Moreover, the components that are diffeomorphic to $\bar{B^3}$ are either (a) convex or (b) capped $\eps$-tubes; the components diffeomorphic to $D^2 \times \Sph^1$ are (c) $\eps$-tubular loops.
\end{prop}


\section{Reduction to sphere or tubular neighbourhood of a curve}\label{sec.BHH}

We still assume that $\bar{K}$ is a compact domain in a complete orientable Riemannian three-manifold $(M,g)$ with nonnegative Ricci curvature. The radii $r_s \ll r_m$ in the following definition have to be thought of as being extremely small.

\begin{defi}\label{def.mg}
A \emph{marble graph} in $\bar{K}$ with string radius $r_s$ and marble radius $r_m$ is a domain of the form $G=\calG_{r_s} (K,\gamma)$ where $(K,\gamma)$ is a $(\bA,\bb)$-controlled configuration such that there exists a finite set $P \subset \bar{K}$ such that
\begin{enumerate}
\item $K$ is the disjoint union of the balls $\bar B_{r_m} (p) \subset \bar{K}$ for $p\in P$.
\item for every $p\in P$, $\gamma\cap \bar B_{10r_m} (p)$ is a union of geodesic segments.\label{mg3}
\end{enumerate}

We say that $G$ is a \emph{marble tree} (respectively a \emph{marble circuit}) if $K\cup \gamma$ is contractible (respectively homotopy equivalent to a circle.) To shorten the proofs below a little bit, we sometimes work without the condition \ref{mg3} and call the resulting object an \emph{almost marble graph}.
\end{defi}

In this section, we show the following results.

\begin{theo}\label{theo.isotopy1}
Every mean convex domain $K_0$ in $\bar{K}$ is mean convex isotopic to a marble graph.
\end{theo}

\begin{theo}\label{theo.isotopy2}
Every marble tree in $\bar{K}$ is mean convex isotopic to a closed geodesic $\eps$-ball and every marble circuit is mean convex isotopic to the solid $\varepsilon$-tubular neighbourhood of a curve and this curve is a core of the marble circuit (which is a solid torus).
\end{theo}

Theorem \ref{theo.sphere} will essentially be a direct consequence of these results while the proofs of our other main results will require further topological arguments in the next section. We start with the proof of Theorem \ref{theo.isotopy1}.

\begin{proof}[Proof of Theorem \ref{theo.isotopy1}]
The proof is essentially the same as in the work of the first author with Haslhofer and Hershkovits \cite{BHH19, BHH21} with only very minor modifications. We therefore try to be brief.

Given a mean convex domain $K_0$ in the fixed compact domain $\bar{K}$, fix $\alpha$ and $\delta$ small enough as well as scales $\bH$ and cap-separation parameter $\Gamma$ large enough and then flow $K_0$ by mean curvature flow with surgery according to the existence result from Theorem \ref{thm.existence}. By Proposition \ref{prop.extenction} this flow will become extinct in finite time and by Proposition \ref{prop.discarded} we know how the discarded components look like. We note that everything happens in the compact domain $\bar{K}$, thus allowing a direct application of Theorem \ref{thm.glue}.

We first claim that for sufficiently small $r_s \ll r_m$ (much smaller than the neck size) all discarded components from this flow are mean convex isotopic to almost marble graphs, compare with Proposition 7.4 of \cite{BHH21} and Lemma 3.6 of \cite{BHH19}. While this is obvious for convex discarded components (that are isotopic to a single marble), for capped $\eps$-tubes and $\eps$-tubular loops we first perform a large number of surgeries (i.e.~replacements of necks by opposing caps as in Definition \ref{def.replacement}) obtaining a domain $K^\sharp$, and then glue back in strings of radius $r_s$ along almost geodesics $\gamma$ using our gluing result from Theorem \ref{thm.glue}. By Proposition \ref{prop.surgeryiso} the resulting domain $\mathcal{G}_{r_s}(K^\sharp,\gamma)$ is mean convex isotopic to the original discarded component. If this is done in such a way that the connected components of $K^\sharp$ have sufficiently small diameter, then these are all very small perturbations of short capped off cylinders and hence isotopic to small balls of radius, say, $10r_m$. Denote by $\{K^\sharp_t\}$ the mean convex isotopy from $K^\sharp$ to this union of balls and by $\{\gamma_t\}$ the family of curves starting with $\gamma$ from above that are then extended such that the end points follow the isotopy $\{K^\sharp_t\}$ by normal motion. Then $\mathcal{G}_{r_s}(K_t^\sharp,\gamma_t)$ is a mean convex isotopy from $\mathcal{G}_{r_s}(K^\sharp,\gamma)$ to an almost marble graph.

We then use a backwards induction scheme described in more detail in \cite{BHH21} that essentially works as follows. At the extinction time $t_k$ there are only discarded components. Each of them is isotopic to an almost marble graph as just described. This gives a mean convex isotopy from $K^{-}_{t_{k}}$ to an almost marble graph (possibly disconnected). We then follow the mean curvature flow back to the previous surgery time $t_{k-1}$ yielding an extension to a mean convex isotopy between $K^{+}_{t_{k-1}}$ and the same almost marble graph. Each component that was discarded at $t_{k-1}$ is also isotopic to an almost marble graph and thus $K^{\sharp}_{t_{k-1}}$ is mean convex isotopic to an almost marble graph (namely the union of all previously collected almost marble graphs). Since $K^{\sharp}_{t_{k-1}}$ has been obtained from $K^{-}_{t_{k-1}}$ by replacing necks by opposing caps, we can follow the exact same strategy as in the paragraph above: first glue back in a string along almost geodesics $\gamma$ and use Proposition \ref{prop.surgeryiso} to obtain a mean convex isotopy from $K^{-}_{t_{k-1}}$ to $\mathcal{G}_{r_s}(K^\sharp_{t_{k-1}},\gamma)$. Next use the isotopy of $K^\sharp_{t_{k-1}}$ to an almost marble graph and extend $\gamma$ following this isotopy by normal motion until all the end points arrive on some marble, yielding a mean convex isotopy of the pre-surgery domain $K^{-}_{t_{k-1}}$ to a (new) almost marble graph (obtained from the old one by gluing in additional strings). We can follow this procedure back in time until we arrive at a mean convex isotopy from the initial domain $K_0$ to an almost marble graph. Finally, we shrink down the marbles from radius $10r_m$ to radius $r_m$ and extend the strings radially in order to obtain a proper marble graph rather than just an almost marble graph.

There is only one additional technical detail that we have ignored in the previous paragraph: if at some point $\gamma$ is extended through regions that are modified by later surgeries, we artificially move it out of these surgery regions. This argument works in the exact same way as in the Euclidean case described in \cite[Lemma 9.4]{BHH21}. The proof is then complete.
\end{proof}

\begin{proof}[Proof of Theorem \ref{theo.isotopy2}]
Again, the proof follows \cite{BHH19, BHH21}. An important ingredient is a rearrangement result, which allows to move curves/strings meeting a fixed marble to meet that marble in points of our choice. Intuitively, we shrink the marble radius to create ``free space'' in which we can move the curves around. Shrinking for example by factor $20$ will create enough space to first move the curves around (in the outer region of the free space created) and then let them be geodesic segments before meeting the marble (in the inner region of the new free space). This procedure makes our control parameters worse, but if we enlarge the whole manifold $(M,g)$ by the same factor, then the control parameters remain the same (while ambient curvature terms decrease) and therefore Theorem \ref{thm.glue} still works the same, showing that without rescaling everything goes through if we also decrease the string radius by factor $20$. This rearrangement result was proved precisely in Euclidean space as Lemma 5.3 of \cite{BHH21} and the exact same argument also works in an ambient manifold using geodesic normal coordinates with origin at the center of the marble where we rearrange curves.

Let us now first consider the case of a \emph{marble tree}. If there is only one marble, we are done. Else, fix a leaf of the tree, i.e. a marble $\bar{B}_1$ that has only one string attached and call $\bar{B}_2$ the unique marble that is connected to $\bar{B}_1$. Say the curve $\gamma_{1,2}$ along which the two marbles have been glued together meets $\bar{B}_2$ at its north pole. By what we have just discussed above, we can move all other curves/strings that meet $\bar{B}_2$ to the southern hemisphere (a step that requires us to shrink both marble and string radius). We then move $\bar{B}_1$ along the curve $\gamma_{1,2}$ closer to $\bar{B}_2$ (possibly violating Point \ref{mg3} of Definition \ref{def.mg} at marble $\bar{B}_1$ for some time along the way) until it has distance $10r_m$ to $\bar{B}_2$. Now the connecting curve $\gamma_{1,2}$ is a geodesic segment. We then use the exponential map based at the center of $\bar{B}_2$ to lift the domain to the tangent space (where $\gamma_{1,2}$ becomes a straight line) and then use the Euclidean marble reduction result from Proposition 5.5 of \cite{BHH21} to completely remove the leaf marble $\bar{B}_1$ and connecting curve $\gamma_{1,2}$ in a mean convex way. By Lemma \ref{lemma.glue}, applying the exponential map to this mean convex isotopy, we find the desired mean convex isotopy in $(M,g)$. We can repeat this step until we are only left with one marble which is our desired $\eps$-ball.

In the case of a \emph{marble circuit}, we first use the above argument to get rid of all leaf marbles. In the resulting marble circuit there are two curves meeting each marble and by the rearrangement argument from the first paragraph of the proof, we can move them in such a way that they meet the marble at antipodal points. Let us now focus on one marble: there are two curves $\gamma_1$ and $\gamma_2$ meeting the marble at antipodal points as radial geodesics. We can then join these geodesics with a geodesic segment in the interior of the marble. We then lift the problem again to the tangent space via the exponential map based at the center of the marble. This brings us back to the Euclidean case (or more precisely, in fact to the model case of a tube along a straight line glued to a ball). By an easy argument (see Theorem 4.1 of \cite{BHH19}), we can find a mean convex isotopy deforming this domain to a solid tube. This way we get rid of the marble and are only left with the tubular neighbourhood of a geodesic where we previously had a marble. Performing this step with all marbles, we find our solid $\varepsilon$-tubular neighbourhood of a curve. It is clear that the curve is a core of this resulting solid torus and thus also of the original marble circuit. The proof of the theorem is complete.
\end{proof}


\section{Proofs of the main results}\label{sec.topology}

We start with the proof of Theorem \ref{theo.sphere} which follows directly from what we have done in the previous sections.

\begin{proof}[Proof of Theorem \ref{theo.sphere}]
Let $(M,g)$ be a complete, orientable Riemannian three-manifold with nonnegative Ricci curvature. Take two elements of $\calM_{H>0}(\Sph^2,M)$, represented by two embedded mean convex two-spheres $\Sigma_1$ and $\Sigma_2$. By Proposition \ref{prop.bounddomain}, $\Sigma_i$ bounds a mean convex domain $\bar{K}_i$ ($i=1,2$) which by Theorems \ref{theo.isotopy1} and \ref{theo.isotopy2} is mean convex isotopic to a small ball $\bar{B}_{\eps_i}(p_i)$ inside $\bar{K}_i$. (Note that all steps are performed inside the compact regions $\bar{K}_i$). Connect $p_1$ and $p_2$ with each other along a curve $\gamma$. Let $0<\eps \leq \eps_i$ be small enough such that \emph{all} balls $\bar{B}_{\eps}(p)$ with $p \in \gamma$ are mean convex. Then we have another mean convex isotopy from $\bar{B}_{\eps_1}(p_1)$ to $\bar{B}_{\eps_2}(p_2)$ by first shrinking the ball to size $\eps$, then moving it along $\gamma$ and finally expanding it again to size $\eps_2$. Combining everything $\bar{K}_1$ is mean convex isotopic to $\bar{K}_2$ and therefore $\Sigma_1$ and $\Sigma_2$ are in the same path component of the moduli space, proving the theorem.
\end{proof}

We are left with proving the two theorems about Heegaard tori. For this reason, we fix a lens space $L(p,q)$ and a metric $g$ on it having nonnegative Ricci curvature. As $L(p,q)$ is compact, we can simply work with $\bar{K}=L(p,q)$. Moreover, we fix once and for all a Heegaard torus $\Sigma^0$ in $L(p,q)$ and denote by $V_1^0$ and $V_2^0$ the two solid tori whose common boundary is $\Sigma^0$. We also fix for each $k\in\{1,2\}$ a core $\gamma_k^0$ of $V_k^0$ that is a controlled curve.

For every $\eps>0$, we denote by $V_k^\eps$ the $\eps$-neighbourhood of $\gamma_k$ and by $T_k^\eps$ its boundary, $k\in\{1,2\}$. We choose $\eps$  small enough so that $V_k^\eps$ is a mean convex domain, and for every $0<\eps'<\eps$ the domain $V_k^{\eps'}$ is mean convex isotopic to $V_k^\eps$. In particular, the path components of $T_1^\eps$ and $T_2^\eps$ in $\calM_{H>0}(p,q)$ are well-defined. 

\begin{prop}\label{prop.mcf}
Let $\Sigma$ be a mean convex Heegaard torus $\Sigma$ and let $V$ be the unique mean convex domain such that $\partial V=\Sigma$. Then there exists $k\in\{1,2\}$ such that $V$ is mean convex isotopic to $V_k^\eps$. Moreover, if $q\not\cong \pm 1 \!\mod p$, then such a $k$ is unique.
\end{prop}

\begin{proof}
Since $\Sigma$ is Heegaard, $V$ is a solid torus. By Bonahon's theorem (cf.~\cite[Th\'eor\`eme 1]{Bo83}), $\Sigma$ is ambient isotopic to $\Sigma^0$, which means that there is a continuous 1-parameter family $\phi_t$ of self diffeomorphisms of $L(p,q)$ such that $\phi_0$ is the identity of $L(p,q)$ and $\phi_1(\Sigma)=\Sigma^0$. We choose $k\in\{1,2\}$ so that $\phi_1(V)=V_k^0$.

By Theorems \ref{theo.isotopy1} and \ref{theo.isotopy2}, $V$ is mean convex isotopic to the solid $\varepsilon$-tubular neighbourhood $V^\eps$ of some controlled curve $\gamma$. It follows that $V^\eps$ is isotopic to $V_k^0$. Hence their cores $\gamma$ and $\gamma_k^0$ are also isotopic. If $\eps$ is chosen small enough, then the tubular neighbourhoods $V^\eps$ and $V_k^\eps$ are mean convex isotopic. Thus $V$ is mean convex isotopic to $V_k^\eps$.

If $q\not\cong \pm 1 \!\mod p$, then by a classical computation (cf.~\cite[Section 6]{BO83}) the curves $\gamma_1^0$ and $\gamma_2^0$ are not homologous, hence not isotopic. This implies that $V_1^\eps$ is not isotopic to $V_2^\eps$.
\end{proof}

We now turn to the proof of Theorem~\ref{theo.mainone}, knowing already that $\calM_{H>0}(p,q)$ has at most two components by Proposition \ref{prop.mcf}.

\begin{proof}[Proof of Theorem~\ref{theo.mainone}]
There are two cases.

\setcounter{case}{0}
\begin{case} 
$q\cong \pm 1 \!\mod p$. In this case we have to prove that $\calM_{H>0}(p,q)$ is path-connected.
\end{case}

In view of Proposition~\ref{prop.mcf}, there only remains to show that $T_1^\eps$ and $T_2^\eps$ are mean convex isotopic to each other. By~\cite[Lemme 4]{Bo83}, $\gamma_1^0$ is isotopic to $\gamma_2^0$. Again, choosing $\eps$ small enough it follows that $T_1^\eps$ is mean convex isotopic to $T_2^\eps$.

\begin{case} \label{case2}
$q\not\cong \pm 1 \!\mod p$. Here we must show that $T_1^\eps$ and $T_2^\eps$ are \emph{not} mean convex isotopic to each other.
\end{case}

We set $H=\mathrm{H}_1(L(p,q),\ZZ) / {\pm 1}$. Each simply closed curve $\gamma$ in $L(p,q)$ defines a class $[\gamma]$ in $H$. As we recalled above, a homology computation shows that the hypothesis $q\not\cong \pm 1 \!\mod p$ implies that $[\gamma_1] \neq [\gamma_2]$.

Let $T$ be a mean convex Heegaard torus in $L(p,q)$ and $V$ be the unique mean convex solid torus such that $\bord V=T$. By Proposition~\ref{prop.mcf}, there exists a unique $k\in\{1,2\}$ such that $V$ is mean convex isotopic to $V_k^\eps$. We set $\phi(T)=[\gamma_k]$.

This defines a continuous map $\phi$ from $\calM_{H>0}(p,q)$ to the discrete space $H$ with the following property: for each $k\in\{1,2\}$, we have $\phi(T_k^\eps)=[\gamma_k]$. This entails that $T_1^\eps$ and $T_2^\eps$ belong to two different path components of $\calM_{H>0}(p,q)$. This completes the proof in Case \ref{case2}, hence that of Theorem~\ref{theo.mainone}.
\end{proof}

Finally, we prove Theorem \ref{theo.maintwo}.

\begin{proof}[Proof of Theorem~\ref{theo.maintwo}]

First note that if $H\geq 0$ on $\Sigma$, then unless $H\equiv 0$, we will \emph{instantaneously} obtain positive mean curvature when starting the mean curvature flow from $\Sigma$. This follows by applying the maximum principle to the evolution equation
\begin{equation*}
\partial_t H = \Delta H + |A|^2 H + \Ric_g(\nu,\nu) H \geq \Delta H + \tfrac{1}{2} H^3.
\end{equation*}
This means that tori with $H\geq 0$ and $H \not \equiv 0$ lie on the boundary of the moduli space $\calM_{H>0}(p,q)$ and therefore do not add to the number of path components.

It remains to consider the case of minimal Heegaard tori, i.e. Heegaard tori with $H \equiv 0$. For such a torus, denote by $V_1$ and $V_2$ the two solid tori bounded by $\Sigma$ and let $\nu$ be the unit normal vector field on $\Sigma$ pointing towards $V_1$. We then consider the stability operator $L_\Sigma \varphi=\Delta_\Sigma \varphi +(\abs{A^\Sigma}^2+\Ric_g(\nu,\nu))\varphi$, where $\varphi: \Sigma \to \RR$. Due to the assumption that $\Ric_g >0$, we know that the largest eigenvalue $\lambda$ of $L_\Sigma$ must be positive, meaning that $\Sigma$ is unstable. We then define for $t\in (-\eps,\eps)$ the family of tori 
\begin{equation*}
\Sigma_t = \exp_{x\in \Sigma} (t\varphi_\lambda(x)\nu(x)),
\end{equation*}
where $\varphi_\lambda$ is a normalised, positive eigenfunction of the eigenvalue $\lambda$. Also consider unit vector fields $\nu_t$ on $\Sigma_t$, continuous in $t$ and such that $\nu_0=\nu$. We obtain 
\begin{equation*}
\frac{\partial}{\partial t} \Big|_{t=0} H_{\Sigma_t} = L_\Sigma \varphi_\lambda = \lambda \varphi_\lambda,
\end{equation*}
and therefore
\begin{equation*}
H_{\Sigma_t} = t \lambda \varphi_\lambda + O(t^2).
\end{equation*}
In particular, for $t>0$, $\Sigma_t$ is mean convex with respect to the normal vector field $\nu_t$ and by our previous arguments $\Sigma_t$ is mean convex isotopic to a tubular neighbourhood of a core $\gamma_1$ of $V_1$. On the other hand, for $t<0$, $\Sigma_t$ is mean convex with respect to $-\nu_t$ and therefore mean convex isotopic to a tubular neighbourhood of a core $\gamma_2$ of $V_2$. This means in particular that if $\calM_{H>0}(p,q)$ has two path components, $\Sigma$ is ($H\geq0$)-isotopic to mean convex Heegaard tori in \emph{both} components and thus $\calM_{H\geq0}(p,q)$ is path-connected.

The theorem now follows by noting that for $\Ric_g>0$, there always exists a minimal Heegaard torus. This follows for example from \cite{KMN} and \cite{LW}.
\end{proof}


\makeatletter
\def\@listi{%
  \itemsep=0pt
  \parsep=1pt
  \topsep=1pt}
\makeatother
{\fontsize{10}{12}\selectfont

}

\printaddress

\end{document}